\documentclass[english]{article}
\usepackage[T1]{fontenc}
\usepackage[latin9]{inputenc}
\usepackage{amsthm}
\usepackage{amsmath}
\usepackage{amssymb}
\usepackage{esint}
\usepackage{verbatim}
\makeatletter
\numberwithin{equation}{section}
\numberwithin{figure}{section}
\newcommand{\lyxaddress}[1]{
\par {\raggedright #1
\vspace{1.4em}
\noindent\par}
}
\theoremstyle{plain}
\newtheorem{thm}{\protect\theoremname}[section]
  \theoremstyle{plain}
  \newtheorem*{thm*}{\protect\theoremname}
  \theoremstyle{plain}
  \newtheorem*{prop*}{\protect\propositionname}
  \theoremstyle{plain}
  \newtheorem{prop}[thm]{\protect\propositionname}
  \theoremstyle{plain}
  \newtheorem*{lem*}{\protect\lemmaname}

\usepackage{babel}

\makeatother

\usepackage{babel}
  \providecommand{\lemmaname}{Lemma}
  \providecommand{\propositionname}{Proposition}
  \providecommand{\theoremname}{Theorem}
\providecommand{\theoremname}{Theorem}

\begin{document}

\title{Quantitative Properties on the Steady States to A Schr$\ddot{o}$dinger-Poisson-Slater
System}

\author{Xiang Changlin}

\maketitle

\lyxaddress{Department of Mathematics and Statistics, PO Box 35 (MaD) FI-40014,
University of Jyv$\ddot{a}$skyl$\ddot{a}$, Finland. E-mail: changlin.c.xiang@jyu.fi}

\begin{abstract}
A relatively complete picture on the steady states of the following
Schr$\ddot{o}$dinger-Poisson-Slater (SPS) system
\[
\begin{cases}
-\Delta Q+Q=VQ-C_{S}Q^{2}, & Q>0\text{ in }\mathbb{R}^{3}\\
Q(x)\to0, & \mbox{as }x\to\infty,\\
-\Delta V=Q^{2}, & \text{in }\mathbb{R}^{3}\\
V(x)\to0 & \mbox{as }x\to\infty.
\end{cases}
\]
 is given in this paper: existence, uniqueness, regularity and asymptotic
behavior at infinity, where $C_{S}>0$ is a constant. To the author's
knowledge, this is the first uniqueness result on SPS system.
\end{abstract}



\section{Introduction}

Due to its importance in various physical frameworks: gravitation,
plasma physics, semiconductor theory, quantum chemistry and so on
(see, e.g. \cite{Soler2003,Lieb1977,Mauser2001} and the reference
therein), the following Schr$\ddot{o}$dinger-Poisson-Slater (SPS)
system in terms of the wave function $\psi:\mathbb{R}^{3}\times[0,T)\to\mathbb{C}$
\[
\begin{cases}
i\frac{\partial\psi}{\partial t}=-\Delta_{x}\psi+V(x,t)\psi+C_{S}|\psi(x,t)|^{2\alpha}\psi, & \lim_{x\to\infty}\psi(x,t)\to0\\
\psi(x,t=0)=\psi_{0}(x),\\
-\Delta_{x}V=\epsilon|\psi|^{2}, & \lim_{x\to\infty}V(x,t)\to0
\end{cases}
\]
has been studied extensively in recent years, see
\cite{Soler2006,Soler2003,Soler2013,Mugnai2004,Mauser2001,Ruiz2006,Soler2004}
for instance. Here $\alpha$, $\epsilon$, $C_{S}$ are constants.,
such as $\epsilon=+1$ (repulsive case), $\epsilon=-1$ (attractive
case) depending on the type of interaction considered. The last
term (the Slater term $|\psi(x,t)|^{2\alpha}\psi$) is usually
considered to be a correction to the nonlocal term $V\psi$, for
example, $\alpha=\frac{1}{3}$, which is called the Slater
correction, or $\alpha=\frac{2}{3}$, which is named as Dirac
correction. The interested reader is recommended to find more
background in the reference, see \cite{Soler2013,Soler2004}.

Based on the fact that the total $mass$
\[
M[\psi]:=\int_{\mathbb{R}^{3}}|\psi(x,t)|^{2}dx
\]
 and the energy functional
\[
E[\psi]=\frac{1}{2}\int_{\mathbb{R}^{3}}|\nabla\psi|^{2}+\frac{\epsilon}{4}\int_{\mathbb{R}^{3}}\left(\frac{1}{4\pi}|x|^{-1}\ast|\psi|^{2}\right)|\psi|^{2}+\frac{C_{S}}{2\alpha+2}\int_{\mathbb{R}^{3}}|\psi|^{2\alpha+2}
\]
 are preserved along the time evolution, the standing wave solutions,
whose interest lies in stability properties stated in terms of the
energy and the mass highlighted in \cite{Lions1982} for example,
obtain a particular importance in the study of the SPS system, i.e.,
solutions of the form
\[
\psi(x,t)=e^{i\lambda t}Q(x)
\]
 with $\lambda>0$ and $Q\in L^{2}(\mathbb{R}^{3})$, which lead us
to the following system

\begin{equation}
\begin{cases}
-\Delta Q+VQ+C_{S}|Q|^{2\alpha}Q=-\lambda Q, & \text{in }\mathbb{R}^{3}\\
-\Delta V=\epsilon|Q|^{2}, & \text{in }\mathbb{R}^{3}\\
V(x)\to0 & \mbox{as }x\to\infty.
\end{cases}\label{eq:SPS system}
\end{equation}

Let $I_{2}=\frac{1}{4\pi}|x|^{-1}$ be the Newtonian potential in
the Euclidean space $\mathbb{R}^{3}$. It is well known that the self-consistent
potential $V$ can be then rewritten explicitly in the form of convolution
by the Newtonian potential $I_{2}$ as
\[
V(x)=\epsilon I_{2}\ast|Q|^{2}(x)=\epsilon\int_{\mathbb{R}^{3}}\frac{|Q(y)|^{2}}{4\pi|x-y|}dy,
\]
so that the SPS system (\ref{eq:SPS system}) is reduced to a singer
nonlinear and nonlocal Schrödinger-type equation

\begin{equation}
-\Delta Q+\lambda Q=-\epsilon\left(I_{2}\ast Q^{2}\right)Q-C_{S}|Q|^{2\alpha}Q,\label{eq:single SPS equation}
\end{equation}
 which is a special case of Schrödinger-Maxwell equations (\cite{Mugnai2004}).

The existence of standing waves has been studied from various perspectives
in the vast mathematical literature. For example, in \cite{Ruiz2006},
the author investigated the existence of critical points of the functional
$E[\psi]+\lambda M[\psi]$ on the Sobolev space $H^{1}(\mathbb{R}^{3})$,
see also the reference therein. Nevertheless, due to the variational
structure of the equation (\ref{eq:single SPS equation}), solutions
of (\ref{eq:single SPS equation}) are usually treated as kind of
critical points of some energy functionals. From a physical point
of view, the most interesting critical points, the so called $steady$
$states$, are minimizers of the problem

\[
I_{M}=\inf\left\{ E[u];u\in H^{1}\left(\mathbb{R}^{3}\right);|u|_{2}=M\right\} .
\]
In the repulsive case $\epsilon=+1,C_{S}<0$, many results have been
known. We only mention some of them. In \cite{reinhard1994}, a negative
answer was given to $\alpha=0$. In \cite{Lions1992}, a positive
answer is given to $\alpha\in(0,1/2)$ with $M>0$ small. In \cite{Soler2004},
as part of its results, a positive answer was obtained to the Slater
correction case: $\alpha=1/3$. Quite recently the authors in \cite{Soler2013}
applied the well known concentration-compactness method to study the
existence of the minimizer of $I_{M}$ systematically for all $\alpha\in(0,2)$,
and a very complete answer on this question was given there. See also
the reference therein to obtain more information on the known result.

There are mainly one obstacle which brings difficulty in the repulsive
case: the variational problem is translation invariant which causes
the loss of compactness of minimizing sequence. Thus the concentration-compactness
method supplies a suitable framework to tackle this problem. However,
in the attractive case $\epsilon=-1,C_{S}>0$, even though the problem
is sill translation invariant but with another good structure which
makes the study easier: unlike in the repulsive case, there is no
competition between the nonlocal term $\iint I_{2}(x-y)|u(x)|^{2}|u(y)|^{2}dxdy$
and the kinetic energy $\int|\nabla u|^{2}$ , thus the powerful rearrangement
method is also available.

In this paper, we shall limit ourselves to the attractive case with
positive contribution given by the Slater term, i.e., $\epsilon=-1,C_{S}>0$.
The exact value of $C_{S}$ is not important but since it can not
be normalized, we shall always leave it in the equation. Moreover,
limited to the uniqueness argument on the minimizer, we will only
consider the case $\alpha=\frac{1}{2}$. Under these considerations
we are able to draw a relatively complete picture about the steady
states: existence, uniqueness, smoothness and asymptotic behavior.
Most of the results are allowed to extend to general cases, which
will be pointed out whenever it is possible. Our main result reads
\begin{thm}
Let $\epsilon=-1,C_{S}>0,\alpha=\frac{1}{2}$, $M>0$. $I_{M}$ is
attained at a minimizing function $Q$ with the following properties:

(1) (radiality and monotonicity) there exists a decreasing function
$Q_{0}:[0,\infty)\to(0,\infty)$ such that $Q(x)=Q_{0}(|x-x_{0}|)$
for some $x_{0}\in\mathbb{R}^{3}$;

(2) (regularity) modulo some scaling transformations, $Q$ satisfies
the Euler-Lagrange equation for some constant $C_{S}$
\[
\begin{cases}
-\Delta Q+Q=\left(\int_{\mathbb{R}^{3}}\frac{1}{4\pi|x-y|}Q(y)^{2}dy\right)Q-C_{S}Q^{2}, & Q>0\text{ in }\mathbb{R}^{3},\\
Q\in H^{1}(\mathbb{R}^{3})\cap C^{\infty}(\mathbb{R}^{3}).
\end{cases}
\]
 (3) (uniqueness) modulo translations, the minimizer $Q$ is unique.
\end{thm}
As mentioned above, the existence result will be established in section
2 via a rearrangement argument combining together with concentration
compactness method.

It is a tough job to study the uniqueness of minimizers or positive
solutions (modulo translation) to both repulsive case and attractive
case of SPS system. Very little is known. Since the energy functional
is not convex, the conventional way to prove uniqueness does not apply.
Another way to show uniqueness is by virtue of the Euler-Lagrange
equation satisfied by minimizers, and ODE technique could be applied
once the radially symmetric property of minimizers is shown. However,
because of the competition between the kinetic term $\int_{\mathbb{R}^{3}}|\nabla u|^{2}$
and the nonlocal term $\int_{\mathbb{R}^{3}}\left(I_{2}\ast u^{2}\right)u^{2}$,
in the repulsive case it is not trivial to show the radiality of minimizers.
See \cite{Vladimir2012Radiality} for some part results on this.

Let us briefly discuss about the uniqueness of positive (radial) solutions
of the SPS equation here. Without considering the nonlocal term, there
has been a lot of results on the uniqueness of the scalar field equation

\[
\:-\Delta u+u=u^{p},\quad u>0\;\text{in }\mathbb{R}^{n},
\]
 and its generalizations, we only mention the paper \cite{Kwong1989Uinqueness}
here. For the SPS system, the difficulty mainly stems from the nonlocal
term, for which requires the global information of the solutions always.
Thus the local method, the commonly applied shooting method of ODE
is not available in general. Under the consideration of the nonlocal
term but without the Slater term, the first result is due to E.H.Lieb
who derived the uniqueness of positive $radially$ $symmetric$ solution
to the Choquard equation
\[
-\Delta u+u=\left(\int_{\mathbb{R}^{3}}\frac{1}{4\pi|x-y|}u(y)^{2}dy\right)u,\quad\text{in }\mathbb{R}^{3},\quad u(x)\to0\;\text{as }|x|\to\infty
\]
in \cite{Lieb1977Uniqueness}. His argument depends heavily on the
special structure of the equation, and seems impossible to extend
to the equation with general Choquard term $\left(\int_{\mathbb{R}^{3}}\frac{1}{4\pi|x-y|}|u(y)|^{p}dy\right)|u|^{p-2}u$
for $p>1$. But his method is allowed to extend to some higher dimensions,
say, $4$ and $5$ dimension, see \cite{Lenzmann2009Uniqueness,Moroz1999Uniqueness}
for example. Quite recently, Lieb's result was proved to be true for
positive solutions by Ma and Zhao in \cite{Mali2010Radiality}.

Contrary to the attractive case, a big difference is shown to the
repulsive case($\epsilon=+1,C_{S}<0$). In \cite{Ruiz2006} Ruiz showed
the following nonexistence and multiplicity result on positive radial
solutions
\begin{thm*}
\cite{Ruiz2006} Let $1<p\le2,\lambda>0$ and consider the equation
\[
\begin{cases}
-\Delta u+u+\lambda\phi u=|u|^{p-1}u, & \text{ in }\mathbb{R}^{3},\\
-\Delta\phi=u^{2}, & \phi\to0\text{ as }|x|\to\infty.
\end{cases}
\]
 Then (1) if $\lambda\ge1/4$, there is no nontrivial solution in
the space $H^{1}\times D^{1}$;

(2) if $\lambda$ small enough, there exist at least two positive
radial solutions.
\end{thm*}
The multiplicity result has also been proved in \cite{Ruiz2005,WeiJC2005}
in case $1<p<11/7$. Note that by setting $u=\frac{1}{\sqrt{\lambda}}v$,
we get equation
\[
\begin{cases}
-\Delta v+v+\phi v=\lambda^{\frac{1-p}{2}}|v|^{p-1}v, & \text{ in }\mathbb{R}^{3},\\
-\Delta\phi=v^{2}, & \phi\to0\text{ as }|x|\to\infty.
\end{cases}
\]
 i.e., $\epsilon=1,C_{S}=-\lambda^{\frac{1-p}{2}}$ in our context.
Thus in general one does not expect uniqueness of positive $radial$
solutions to the repulsive case of SPS system for large value of $C_{S}$.
However, we still don't know whether the uniqueness of minimizers
if it exist hold or not.

In section 3, we shall follow Lieb's idea to show the uniqueness of
minimizers of $I_{M}$ by virtue of the equation
\[
-\Delta u+u=\left(\int_{\mathbb{R}^{3}}\frac{1}{4\pi|x-y|}u(y)^{2}dy\right)u-C_{S}u^{2},\quad u=u(|x|)>0\;\text{in }\mathbb{R}^{3}.
\]
 An interesting question is to extend the uniqueness result to positive
solution of the equation under the mild condition that $u\to0$ as
$x\to\infty$. It is also expected to extend our uniqueness result
to the problem with a general Slater term.

In the last section, we give a study on asymptotic behavior of the
minimizer, which states that
\begin{thm}
\label{thm:asymptotic behavior}Let $Q\in H^{1}(\mathbb{R}^{3})$
be a positive radially symmetric solution to the equation
\[
\begin{cases}
-\Delta Q+Q-\left(I_{2}\ast Q^{2}\right)Q+C_{S}Q^{2}=0, & \text{in }\mathbb{R}^{3}\\
Q(|x|)\to0 & \text{as }|x|\to\infty.
\end{cases}
\]
 Then
\[
\lim_{|x|\to\infty}Q(x)|x|^{1-\alpha/2}e^{|x|}\in(0,\infty),
\]
with
\[
\alpha=\frac{1}{4\pi}\int_{\mathbb{R}^{3}}|Q|^{2}dx.
\]

\end{thm}
A more general result on the asymptotic behavior of solutions to SPS
system will also be given there in the last section.

All the notations in the paper are standard, constants will be denoted
by $C,C_{M},...$ which may be different from line to line.

\section{Existence }

Without misunderstanding, in this section all the integrals will be
taking in $\mathbb{R}^{3}$ unless specified. Our approach is variational.
Since the value of $C_{S}$ does not enter the play, we shall always
assume that $C_{S}=1$ in the following.

Consider the following minimizing problem: let $M>0$ and $\mathcal{A}_{M}=\{u\in H^{1}\left(\mathbb{R}^{3}\right);|u|_{L^{2}}=M\}$,
define

\begin{equation}
I_{M}=\inf\left\{ E(u);u\in\mathcal{A}_{M}\right\} ,\label{eq:minimizing problem}
\end{equation}
where
\begin{equation}
E(u)=\frac{1}{2}\int|\nabla u|^{2}-\frac{1}{4}\int\left(I_{2}\ast u^{2}\right)u^{2}+\frac{1}{3}\int|u|^{3}.\label{eq:Energy functional}
\end{equation}

The first observation of $I_{M}$ is that
\begin{prop*}
$0>I_{M}>-\infty$, and minimizing sequences are bounded.
\end{prop*}
In fact, for any $u\in H^{1}$, the Hardy-Littlewood-Sobolev inequality
implies that there exists a positive constant $C$ such that
\[
\int_{\mathbb{R}^{3}}\left(I_{2}\ast u^{2}\right)u^{2}\le C|u|_{\frac{12}{5}}^{4},
\]
 thus combining together with Sobolev inequality we have
\begin{align*}
\int_{\mathbb{R}^{3}}\left(I_{2}\ast u^{2}\right)u^{2} & \le C|u|_{\frac{12}{5}}^{4}\le C\left(|u|_{2}^{3/4}|u|_{6}^{1/4}\right)^{4}\\
 & \le C|u|_{2}^{3}|\nabla u|_{2}\\
 & \le C|u|_{2}^{6}+\int|\nabla u|^{2},
\end{align*}
 therefore, for all $u\in\mathcal{A}_{M}$ there holds
\[
E(u)\ge\frac{1}{4}\int|\nabla u|^{2}-CM^{6}\ge-CM^{6}>-\infty,
\]
 furthermore, the above inequality implies that any minimizing sequence
of $I_{M}$ is a bounded sequence in $H^{1}(\mathbb{R}^{3})$. Easy
to note that the Slater term could be replaced by $C_{S}|Q|^{2\alpha+2}$
for any $\alpha\in(0,2)$.

$I_{M}<0$ is consequence of scaling property of the energy functional
$E$. For any $t>0$, let $u^{t}(x)=t^{3/2}u(tx)$ so that $|u^{t}|_{2}=|u|_{2}$,
then
\[
E(u^{t})=\frac{t^{2}}{2}\int|\nabla u|^{2}-\frac{t}{4}\int\left(I_{2}\ast u^{2}\right)u^{2}+\frac{t^{3/2}}{3}\int|u|^{3}.
\]
 Easily to see that
\[
\inf_{t>0}E(u^{t})<0,
\]
 for any nonzero function $u$, from which follows that
\[
I_{M}=\inf\left\{ \inf_{t>0}E(u^{t});u\in\mathcal{A}_{M}\right\} <0.
\]

We remark that only if $2>\alpha>1/3$, we can show the same thing
with the Slater term $C_{S}|Q|^{2\alpha+2}$.

The following proposition is easy but crucial, which will provide
sufficient and necessary condition for the compactness of minimizing
sequences of $I_{M}$. The index $\alpha=1/2$ plays a role here.
For other $\alpha$'s, there are some argument difficulty.
\begin{prop}
\label{prop: IM decreasing}$I_{M}<I_{\alpha}$ for any $0<\alpha<M$.
Furthermore, $I_{M}=M^{6}I_{1}$.\end{prop}
\begin{proof}
The transform $u\mapsto u_{M}=M^{4}u(M^{2}x)$ is a 1-1 mapping from
$\mathcal{A}_{1}$ to $\mathcal{A}_{M}$, direct computation yields
that
\[
E(u_{M})=M^{6}E(u),
\]
hence
\[
I_{M}=\inf\left\{ E(u);u\in\mathcal{A}_{M}\right\} =\inf\left\{ M^{6}E(u);u\in\mathcal{A}_{1}\right\} =M^{6}I_{1}.
\]
 Since $-\infty<I_{1}<0$, $I_{M}$ is strictly decreasing with respect
to $M$. This gives the proof.
\end{proof}
We are now on the position to show the existence result.
\begin{thm}
(Existence) For any $M>0$ there exists a minimizer of $I_{M}$. Moreover,
for any minimizing sequence $\{u_{k}\}$, the rearrangement sequence
$\{u_{k}^{*}\}$ is also a minimizing sequence which contains a subsequence
converging in $ $$H^{1}\left(\mathbb{R}^{3}\right)$, and the minimizer
is nonnegative and nonincreasing radially symmetric. \end{thm}
\begin{proof}
The existence proof is standard, and we just give it for the reader's
convenience.

Suppose that $\{u_{k}\}\subset\mathcal{A}_{M}$ is a minimizing sequence
so that $E(u_{k})\to I_{M}$. As observed above, this sequence is
bounded in $H^{1}$. Let $u_{k}^{*}$ be the nonincreasing symmetric
rearrangement of $u_{k}$, then $u_{k}^{*}\in H^{1}(\mathbb{R}^{3})$
is radially symmetric, nonnegative, nonincreasing function satisfying
\begin{align*}
 & \int|\nabla u_{k}^{*}|^{2}\le\int|\nabla u_{k}|^{2},\\
 & \int\left(I_{2}\ast u_{k}^{*2}\right)u_{k}^{*2}\ge\int\left(I_{2}\ast u_{k}^{2}\right)u_{k}^{2},\\
 & \int|u_{k}^{*}|^{p}=\int|u_{k}|^{p}\quad\text{ for }p=2,3,
\end{align*}
(see \cite{LiebAnalysis} for a proof), thus it follows that $u_{k}^{*}\in\mathcal{A}_{M}$
and
\[
E(u_{k}^{*})\le E(u_{k}),
\]
hence $\{u_{k}^{*}\}$ is also a minimizing sequence and thus bounded.

Therefore, we can always assume that the minimizing sequence $\{u_{k}\}$
is nonincreasing nonnegative and radially symmetric. Thus there exists
a function $u\in H^{1}(\mathbb{R}^{3})$ so that up to a subsequence
there has
\begin{align*}
u_{k}\rightharpoonup u & \text{ in }H^{1}(\mathbb{R}^{3}),\\
u_{k}\to u & \text{ in }L_{loc}^{2}(\mathbb{R}^{3})\\
u_{k}\to u & \text{ a.e.}x\in\mathbb{R}^{3}.
\end{align*}
It is easy to see that $u$ is a nonincreasing nonnegative and radially
symmetric function. Moreover, by the compact embedding theorem of
\cite{Strauss1977Compactness} there holds that
\[
u_{k}\to u\text{ in }L^{p}(\mathbb{R}^{3}),\text{ for all }2<p<6.
\]

Let $v_{k}=u_{k}-u$, it follows from the above strong convergence
in $L^{p}$ space that
\begin{align*}
 & \int|\nabla u_{k}|^{2}=\int|\nabla v_{k}|^{2}+\int|\nabla u|^{2}+o(1)\\
 & \int_{\mathbb{R}^{3}}\left(I_{2}\ast u_{k}^{2}\right)u_{k}^{2}=\int_{\mathbb{R}^{3}}\left(I_{2}\ast u^{2}\right)u^{2}+o(1),\\
 & \int|u_{k}|^{p}=\int|u|^{p}+o(1),
\end{align*}
 where $o(1)\to0$ as $k\to\infty$. Therefore
\begin{equation}
E(u_{k})=\frac{1}{2}\int|\nabla v_{k}|^{2}+E(u)+o(1).\label{eq: dichotomy inequality}
\end{equation}
By passing to limit $k\to\infty$, we obtain that
\[
E(u)\le I_{M}<0,
\]
which implies $u\ne0$. Thus we only need to show that $u\in\mathcal{A}_{M}$,
i.e. $|u|_{2}=M$.

However, by the Fatou's lemma $|u|_{2}\le\liminf_{k}|u_{k}|_{2}=M$.
Suppose that $u\notin\mathcal{A}_{M}$, then $\exists\:\alpha\in(0,M)$
s.t. $|u|_{2}=\alpha$. Thus $E(u)\ge I_{\alpha}>I_{M}$ according
to Proposition \ref{prop: IM decreasing}, which contradicts to $E(u)\le I_{M}$.
Thus $|u|_{2}=M$. This shows that $u$ is a minimizer of $I_{M}$.

Furthermore, since $u_{k}$ converges weakly to $u$ in $L^{2}$ and
also the $L^{2}$ norm is preserved, we see that $u_{k}\to u$ in
$L^{2}$. Also, by the equation (\ref{eq: dichotomy inequality})
we get that $\int|\nabla v_{k}|^{2}\to0$ as $k\to\infty$, i.e.,
$\nabla u_{k}\to\nabla u$ in $L^{2}$. Hence $u_{k}\to u$ in $H^{1}(\mathbb{R}^{3})$.

This completes the proof.
\end{proof}

\section{Smoothness and Uniqueness}

In this section, we show that any minimizer of $I_{M}$ is $C^{\infty}$
smooth and modulo translations, it is unique. Remember that up to
scaling, minimizers of $I_{M}$ satisfies the following Euler-Lagrange
equation

\begin{equation}
\begin{cases}
-\Delta Q+Q-\left(I_{2}\ast Q^{2}\right)Q+C_{S}Q^{2}=0, & \text{in }\mathbb{R}^{3}\\
Q=Q(|x|)\ge0,\\
Q\in H^{1}\left(\mathbb{R}^{n}\right)
\end{cases}\label{eq:SPS-Equation}
\end{equation}
Essentially, all results in this section are consequences of equation
theory. Thus they are applicable to solutions of SPS equation, not
only minimizers of $I_{M}$. First we derive the following regularity
result, which is only a consequence of the regularity theory of elliptic
equations, see,.e.g. \cite{GilbargTrudinger}.
\begin{thm}
$u\in C^{\infty}(\mathbb{R}^{3})$ with bounded derivatives of all
orders. \end{thm}
\begin{proof}
Written by $V=I_{2}\ast Q^{2}$, it follows from the Riesz potential
theory (see \cite{Stein1970SingularIntegral}) that $I_{2}$ is a
bounded linear operator from $L^{p}(\mathbb{R}^{3})$ to $L^{q}(\mathbb{R}^{3})$
for any $1<p<\frac{3}{2}$ and $q=\frac{3p}{3-2p}$, thus $V\in L^{q}(\mathbb{R}^{3})$
for any $3<q<\infty$ since $Q^{2}\in L^{1}\cap L^{3/2}(\mathbb{R}^{3})$,
thus $VQ\in L^{r}(\mathbb{R}^{3})$ for all $r\in(6/5,6)$. As $Q^{2}\in L^{1}\cap L^{3}(\mathbb{R}^{3})$
we see that $VQ+C_{S}Q^{2}\in L^{3}(\mathbb{R}^{3})$. Hence by the
Calderon-Zygmund theory, $Q\in W^{2,3}(\mathbb{R}^{3})$, which implies
that $Q\in C^{\alpha}(\mathbb{R}^{3})$ for all $0<\alpha<1$, the
Hölder space of order $\alpha$.

Hence $Q^{2}\in C^{\alpha}(\mathbb{R}^{3})$ for all $0<\alpha<1$.
Since $-\Delta V=Q^{2}$, Schauder's theory implies that $V\in C^{2,\alpha}(\mathbb{R}^{3})$
for all $0<\alpha<1$. (This could be firstly done locally, say, on
balls $B(x,1)$, but since the estimates does not depend on the center
$x$, we get $C^{2,\alpha}$ estimate on the whole space) Again, we
find that $VQ+Q^{2}$ is Hölder continuous, hence $Q\in C^{2,\alpha}(\mathbb{R}^{3})$
for all $0<\alpha<1$. In general, if $Q\in C^{k,\alpha}(\mathbb{R}^{3})$,
then $V\in C^{2+k,\alpha}(\mathbb{R}^{3})$ for all $0<\alpha<1$,
which ensures that $VQ+Q^{2}\in C^{k,\alpha}(\mathbb{R}^{3})$, and
apply Schauder's theory again to see that $Q\in C^{2+k,\alpha}(\mathbb{R}^{3})$.
This finishes the proof. \end{proof}
\begin{thm}
There exists at most one positive solution of the Schrodinger-Poisson-Slater
equation (\ref{eq:SPS-Equation}). \end{thm}
\begin{proof}
We follow the method of \cite{Lieb1977Uniqueness}. Since $Q\ge0$
and $Q\ne0$, rewriting the equation as
\[
-\Delta Q+c(x)Q=\left(I_{2}\ast Q^{2}\right)Q\ge0\quad\text{ in }\mathbb{R}^{3},
\]
where $c(x)=1+C_{S}Q\ge0$ for all $x\in\mathbb{R}^{3}$, then the
strong maximum principle and Hopf lemma of elliptic equations (see
\cite{GilbargTrudinger} for instance) implies that $Q>0$ and $Q'(|x|)=\frac{\partial Q}{\partial\nu}(x)<0$
for $|x|>0$.

As in \cite{Lieb1977Uniqueness}, applying Newton's theorem which
states that
\[
-\left(I_{2}\ast Q^{2}\right)(r)=\int_{0}^{r}K(r,s)Q^{2}(s)ds-\int_{0}^{\infty}Q^{2}(s)sds=:A_{Q}(r)-\lambda_{Q},
\]
where
\[
K(r,s)=s\left(1-\frac{s^{n-2}}{r^{n-2}}\right)\ge0
\]
 for $r\ge s$, the equation can be read as
\[
-\left(Q''+\frac{n-1}{r}Q'\right)+A_{Q}(r)Q+C_{S}Q^{2}=\left(\lambda_{Q}-1\right)Q.
\]

We claim that $\lambda_{Q}-1>0$. In fact, since $Q\ge0\ge Q'$, by
multiplying $Q$ on both sides of the equation and integrating from
0 to $\infty$, one derives that
\[
\left(\lambda_{Q}-1\right)\int_{0}^{\infty}Q^{2}=\int_{0}^{\infty}Q'^{2}-\frac{n-1}{r}Q'Q+A_{Q}(r)Q^{2}+Q^{3}dr>0.
\]
By a simple scaling argument ($Q(r)=\mu^{2}P(\mu r)$, $\mu^{2}=\lambda_{Q}-1$
), we only need to consider the initial value problem with normalized
coefficient. The special power of Slater term is crucial here.
\[
\begin{cases}
-\left(Q''+\frac{n-1}{r}Q'\right)+A_{Q}(r)Q+C_{S}Q^{2}=Q, & r>0,\\
Q(0)>0,\\
Q'(0)=0.
\end{cases}
\]
We are now able to prove the theorem.

Suppose on the contrary $Q,R$ are two such solutions. Then by
local uniqueness result of ODE, we get that if $Q\ne R$ then
$Q(0)\ne R(0)$. So we may assume that $Q(0)>R(0)$, it follows by
continuity that $Q>R$ in a neighborhood of zero. We use Sturm-type
comparison argument as follows. Multiplying the equation of $Q$ by
$R$ and vice vase, and letting $S=Q'R-QR'$, then we get from the
difference of the two equations that
\[
\left(r^{n-1}S\right)'=r^{n-1}\left(A_{Q}-A_{R}\right)QR+r^{n-1}\left(Q-R\right)QR.
\]
 Since $Q>R$ in a neighborhood of 0, we see that $\left(r^{n-1}S\right)'>0$
in some maximum interval $(0,r^{*})\subset(0,\infty)$ s.t. $\left(r^{n-1}S\right)'>0$
in $(0,r^{*})$ and $\left(r^{n-1}S\right)'(r^{*})=0$. We shall show
that $r^{*}=\infty$. Otherwise, suppose that $r^{*}<\infty$, then
$\left(r^{n-1}S\right)'>0$ in $(0,r^{*})$ and $S(0)=0$ implies
that $S>0$ in $(0,r^{*})$ . Equivalently, i.e.,
\[
\left(\frac{Q}{R}\right)'=\frac{S}{R^{2}}>0\quad in\;(0,r^{*}).
\]
 Hence
\[
Q>\frac{Q(0)}{R(0)}R>R\quad in\;(0,r^{*}].
\]
 But this turns out that
\[
A_{Q}(r^{*})>A_{R}(r^{*}),
\]
 then $\left(r^{n-1}S\right)'(r^{*})>0$, contradiction to the assumption.
Hence $r^{*}=\infty$, and so $Q>R$ and $A_{Q}-A_{R}>0$ in $(0,\infty)$.
However, this is also impossible, since we have
\begin{eqnarray*}
0<\int_{0}^{\infty}r^{n-1}\left(A_{Q}-A_{R}+Q-R\right)QR & = & \int_{0}^{\infty}\left(r^{n-1}Q'\right)'R-\left(r^{n-1}R'\right)'Qdr=0.
\end{eqnarray*}
Thus $Q\equiv R$. This finishes the proof.
\end{proof}

\section{Asymptotic behavior }

In this last section, we study the asymptotic behavior of solutions.
There should be some results already in the literature, but since
we can not give such a reference, and also under the consideration
that it would be useful for a further study in the future, we'd like
to give a precise enough result here. As both repulsive and attractive
case of SPS system with different Slater terms are interesting, we'd
better to allow our result to have some generality.

Theorem \ref{thm:asymptotic behavior} becomes a special case of the
following one
\begin{thm}
\label{thm:Asymptotic behavior}Let $Q$ be a positive radially symmetric
solution to the equation
\[
\begin{cases}
-\Delta Q+Q=\epsilon\left(I_{2}\ast Q^{2}\right)Q+f(Q)Q, & \text{in }\mathbb{R}^{3}\\
Q(|x|)\to0 & \text{as }|x|\to\infty.
\end{cases}
\]
where, $\epsilon=\pm1$, and for some $\beta>0$, $f$ satisfies the
condition that
\[
\lim_{t\to0+}\frac{f(t)}{t^{\beta}}=0.
\]
Then
\[
\lim_{|x|\to\infty}Q(x)|x|^{1-\epsilon\alpha/2}e^{|x|}\in(0,\infty),
\]
with
\[
\alpha=\frac{1}{4\pi}\int_{\mathbb{R}^{3}}|Q|^{2}dx.
\]

\end{thm}
A typical example of $f$ is given by $f(t)=\sum_{i}c_{i}|t|^{\beta_{i}-1}t$
for some constant $c_{i}$ and some positive constants $\beta_{i}$,
especially, $f(t)=-C_{S}|t|^{2\alpha}t$ coincides with the Slater
term.

In case $f\equiv0$, this result has been contained in \cite{Schaftingen2013}
Proposition 6.5, where the authors proved much more general results.
Their method can be easily applied to our theorem, but for the reader's
convenience, we shall give a sketch of proof here. From the proof
the readers will see that the last term $f(Q)Q$ is in essence a perturbation
of the nonlocal term and thus nothing changes. There are still rooms
to relax the restriction on $f$. In the same spirit, our theorem
can be extended to general Choquard equation and higher dimensions
as well, but which is not our interest here.

We divide the proof into several propositions. Hereafter we write
$V=I_{2}\ast Q^{2}$. Obviously $V(x)=V(|x|)>0$ and $V(r)\to0$ as
$r\to\infty$. The first proposition is about exponential decay of
$Q$.
\begin{prop}
$Q(r)=O(e^{-r/2})$ as $r\to\infty$. \end{prop}
\begin{proof}
The proof is based on an easy comparison argument. Let $G(x)=|x|^{-1}e^{-|x|/2}$
so that $-\Delta G+G/2=0$ for $x\ne0$. Since $Q(x)=o(1)$ as $|x|\to\infty$,
for $|x|$ large enough, $\epsilon V(x)-f(Q(x))\le1/2$, thus $-\Delta Q+Q/2\le0$
for $|x|$ large. Therefore by multiplying $Q$ to the equation of
$G$, integrating on the exterior ball $B_{r}^{c}(0)$ and vice vesa,
we get that $Q'G-QG'\le0$ as $r\to\infty$, i.e., $Q/G$ is a nonincreasing
function, hence $Q(r)\le(Q(r_{0})/G(r_{0}))G(r)$ for all $r>r_{0}$
with $r_{0}$ large enough. This gives the proof.
\end{proof}
Secondly we estimate the nonlocal term $V$, which, unlike the solution
$Q$, decays polynomially at infinity and gives contribution to the
decay of $Q$.
\begin{prop}
$V(r)=\frac{\alpha}{r}+O(r^{-2})$ and $V'(r)=-\frac{\alpha}{r^{2}}+O(r^{-3})$
as $r\to\infty$, where $\alpha=\frac{1}{4\pi}\int_{\mathbb{R}^{3}}|Q|^{2}dx$.\end{prop}
\begin{proof}
This is rooted in the fact that $-\Delta V=Q^{2}$. Since $Q$ is
radially symmetric, so is $V$, we have that
\[
\left(r^{2}V'(r)\right)'=-r^{2}Q^{2}(r)\le0,
\]
 which implies that $r^{2}V'(r)$ is nonincreasing and hence $V'\le0$
since $V'(0)=0$ (by smoothness of $V$). Therefore
\[
V'(r)=-r^{-2}\int_{0}^{r}s^{2}Q^{2}(s)ds=-\frac{1}{4\pi r^{2}}\int_{B_{r}(0)}Q^{2}dx,
\]
 which implies that
\[
V(r)=\int_{r}^{\infty}\frac{1}{4\pi s^{2}}\left(\int_{B_{s}(0)}Q^{2}dx\right)ds=\frac{1}{4\pi r}\int_{\mathbb{R}^{3}}Q^{2}dx-\int_{r}^{\infty}\frac{1}{4\pi s^{2}}\left(\int_{|x|\ge s}Q^{2}dx\right)ds.
\]
 According to the exponential decay of $Q$ at infinity we see that
$\int_{r}^{\infty}\frac{1}{4\pi s^{2}}\left(\int_{|x|\ge s}Q^{2}dx\right)ds=O(r^{-2})$.
This finishes the proof.
\end{proof}
Finally to prove the asymptotic behavior of $Q$ at infinity we apply
the following lemma which given by \cite{Schaftingen2013} Lemma 6.4
\begin{lem*}
Let $\rho\ge0$ and $W\in C^{1}((\rho,\infty),\mathbb{R})$. If
\[
\lim_{s\to\infty}W(s)>0
\]
and for some $\beta>0$
\[
\lim_{s\to\infty}W'(s)s^{1+\beta}=0,
\]
 then there exists a nonnegative radial function $v:\mathbb{R}^{N}\backslash B_{\rho}\to\mathbb{R}$
such that
\[
-\Delta v+Wv=0\quad\text{in }\mathbb{R}^{N}\backslash B_{\rho},
\]
 and for some $\rho_{0}\in(\rho,\infty)$,
\[
\lim_{|x|\to\infty}v(x)|x|^{\frac{N-1}{2}}\exp\left(\int_{\rho_{0}}^{|x|}\sqrt{W}\right)=1.
\]

\end{lem*}
Our theorem now is a consequence of the above lemma and the previous
estimates on $V$ and $Q$.

\textbf{Proof of theorem \ref{thm:Asymptotic behavior}. }Let $W=1-\epsilon V-f(Q),$
from the estimates on $V$ and $Q$ one sees $f(Q(x))$ owns exponential
decay by the assumption on $f$. Thus $W=1-\frac{\epsilon\alpha}{r}+O(r^{-2})$
as $r\to\infty$ with $\alpha=\frac{1}{4\pi}\int_{\mathbb{R}^{3}}|Q|^{2}dx$
and $\lim_{s\to\infty}W'(s)s^{1+1/2}=0$. Therefore, according the
above lemma, for $\rho$ large enough, there exists a nonnegative
radial function $v:\mathbb{R}^{3}\backslash B_{\rho}\to\mathbb{R}$
such that
\[
-\Delta v+Wv=0\quad\text{in }\mathbb{R}^{3}\backslash B_{\rho},
\]
 and for some $\rho_{0}\in(\rho,\infty)$,
\[
\lim_{|x|\to\infty}v(x)|x|^{\frac{N-1}{2}}\exp\left(\int_{\rho_{0}}^{|x|}\sqrt{W}\right)=1.
\]
 We show that $Q(x)=cv(x)$ for some constant $c>0$. In fact, since
both $Q$ and $v$ solves the same equation in $\mathbb{R}^{3}\backslash B_{\rho}$,
we see that
\[
\int_{\mathbb{R}^{N}\backslash B_{\rho}}(-\Delta Q+WQ)vdx=\int_{\mathbb{R}^{N}\backslash B_{\rho}}\left(\nabla Q\nabla v+WQv\right)dx+4\pi\rho^{2}Q'(\rho)v(\rho)
\]
 and
\[
\int_{\mathbb{R}^{N}\backslash B_{\rho}}(-\Delta v+Wv)Qdx=\int_{\mathbb{R}^{N}\backslash B_{\rho}}\left(\nabla Q\nabla v+WQv\right)dx+4\pi\rho^{2}Q(\rho)v'(\rho).
\]
 Hence $Q'v=v'Q$, which implies that $(v/Q)'=0$. Hence $Q=cv$ for
some positive constant $c$. Therefore
\[
\lim_{|x|\to\infty}Q(x)|x|\exp\left(\int_{\rho_{0}}^{|x|}\sqrt{W}\right)\in(0,\infty).
\]
Substituting $W=1-\frac{\epsilon\alpha}{r}+O(r^{-2})$ one gets that
\[
\lim_{|x|\to\infty}Q(x)|x|^{1-\epsilon\alpha/2}e^{|x|}\in(0,\infty).
\]
 This finishes the proof.

\bibliographystyle{plain}
\bibliography{Schrodinger-Poisson-Slater}

\end{document}